\documentclass[smallextended]{svjour3}

\usepackage[english]{babel}
\usepackage[T1]{fontenc}
\usepackage[utf8]{inputenc}
\usepackage{amssymb}
\usepackage{amsmath}
\usepackage{amsfonts}
\usepackage{amscd}
\usepackage{color} 
\usepackage{hyperref}
\usepackage[active]{srcltx}

\hypersetup{
  pdftitle={Reverse Khas'minskii Condition},
  pdfauthor={Alberto Giulio Setti, Daniele Valtorta},
  pdfsubject={Potenziali di Evans su varietà paraboliche},
  pdfkeywords={Setti 7i evans potential potentials parabolic manifold tesi matematica}
  pdfpagelayout=SinglePage,
  pdfpagemode=UseOutlines}

\newcommand{\R}{\mathbb{R}}

\newcommand{\D}{\mathcal{D}}

\newcommand{\norm}[1]{\left\|#1\right\|}

\newcommand{\ps}[2]{\left\langle#1\middle\vert#2\right\rangle}

\newcommand{\ton}[1]{\left(#1\right)}
\newcommand{\abs}[1]{\left|#1\right|}

\newcommand{\ka}{\mathcal{K}}
\newcommand{\E}{\mathcal{E}}

\newcommand{\cp}{\operatorname{Cap_p}}
\newcommand{\Ka}{Khas'minskii }

\newtheorem{prop}{Proposition}[section]
\newtheorem{teo}[prop]{Theorem}
\newtheorem{deph}[prop]{Definition}

\newtheorem{rem}[prop]{Remark}

\begin{document}

\title{Reverse Khas'minskii condition
\thanks{Special thanks go to my advisor, prof. Alberto Giulio Setti, whose assistence has proven invaluable in writing this paper}}

\date{\today}

\author{Daniele Valtorta}
\institute{D. Valtorta \at \\ Dipartimento di Matematica\\
Universit\`a degli studi di Milano\\
via Saldini 50\\
20133 Milano, Italy, EU\\
tel. +393492972974\\
\email{danielevaltorta@gmail.com}
}

\begin{abstract}
The aim of this paper is to present and discuss some equivalent characterizations of $p$-parabolicity for complete Riemannian manifolds in terms of existence of special exhaustion functions. In particular, \Ka in \textit{Ergodic properties of recurrent diffusion prossesses and stabilization of solution to the Cauchy problem for parabolic equations} (Theor. Prob. Appl., 5 no.2, 1960) proved that if there exists a $2$-superharmonic function $\ka$ defined outside a compact set on a complete Riemannian manifold $R$ such that $\lim_{x\to \infty} \ka(x)=\infty$, then $R$ is $2$-parabolic, and Sario and Nakai in \textit{Classification theory of Riemann surfaces} (1970) were able to improve this result by showing that $R$ is $2$-parabolic if and only if there exists an Evans potential, i.e. a $2$-harmonic function $E:R\setminus K \to \R^+$ with $\lim_{x\to \infty} \E(x)=\infty$. In this paper, we will prove a reverse \Ka condition valid for any $p>1$ and discuss the existence of Evans potentials in the nonlinear case.



\end{abstract}

\maketitle

Given a complete Riemannian manifold $R$, we say that $R$ is $p$-parabolic if every compact subset $K\subset R$ has $p$-capacity zero, or equivalently if every bounded below $p$-subharmonic function is constant. In the following we briefly recall some definitions and results relative to $p$-capacity and $p$-harmonic functions. Some good references for this introductory part are \cite{3} and \cite{20} (for the $p=2$ case only). Note that \cite{3} works on $\R^n$, but from the proofs it is quite clear that all the local results extend also to generic Riemannian manifolds.

This paper is dedicated to characterize the $p$-parabolicity of a complete Riemannian manifold through the so-called \Ka condition, answering in the affermative to a problem raised, e.g., in \cite{99} pag 820.

In particular we will prove that a complete Riemannian manifold $R$ is $p$-parabolic if and only if for any $p$-regular compact set (for example for many compact set with smooth boundary) there exists a $p$-superharmonic function $f:R\setminus K\to \R^+$ with $f|_{\partial K}=0$ and $\lim_{x\to \infty} f(x)=\infty$. The \Ka condition is discussed in \cite{20} and in \cite{99}; the latter article provides also some other equivalent characterizzations of $p$-parabolicity and applications of the \Ka condition.

In the following $\D_p(f)$ will denote its $p$-Dirichlet integral, i.e.
\begin{gather*}
\D_p(f)\equiv \int_{\Omega} \abs{\nabla f}^p dV
\end{gather*}
where $\Omega$ is the domain of the function $f$. $W^{1,p}(\Omega)$ stands for the standard Sobolev space, while $L^{1,p}(\Omega)$ is the so-called Dirichlet space, i.e. the space of functions in $W^{1,p}_{loc}(\Omega)$ with finite $p$-Dirichlet integral. Hereafter, we assume that $R$ is a complete smooth noncompact Riemannian manifold without boundary with metric tensor $g_{ij}$ and volume form $dV$.

\begin{deph}
 Given a compact set and an open set $K\subset\Omega\subset R$, we define
\begin{gather*}
 \operatorname{Cap_p}(K,\Omega)\equiv \inf_{\varphi \in C^\infty_c(\Omega), \ \varphi(K)=1 }\int_\Omega \abs{\nabla \varphi}^p dV
\end{gather*}
If $\Omega=R$, then we set $\cp(K,R)\equiv \cp(K)$.\\
By a standard density argument for Sobolev spaces, the definition is unchanged if we allow $\varphi -\psi\in W^{1,p}_0(\Omega\setminus K)$, where $\psi$ is a cutoff function with support in $\Omega$ and equal to $1$ on $K$.
\end{deph}
By definition, $R$ is $p$-parabolic if and only if $\cp(K)=0$ for every $K\in R$, or equivalently if there exists a compact set with nonempty interior $\tilde K$ with $\cp(\tilde K)=0$.
\begin{deph}
 A real function $h$ defined on an open $\Omega\subset R$ is said to be $p$-harmonic if $h\in W^{1,p}_{loc}(\Omega)$ and $\Delta_p h = 0$ in the weak sense, i.e.
\begin{gather*}
 \int_\Omega \abs{\nabla h}^{p-2}\ps{\nabla h }{\nabla \phi }dV=0 \ \ \ \forall \ \phi \in C^\infty_c(\Omega)
\end{gather*}
The space of $p$-harmonic functions on an open set $\Omega$ is denoted by $H_p(\Omega)$.
\end{deph}
We recall that $p$-harmonic functions are always continuous (in fact, they are $C^{1,\alpha}(\Omega)$) and they are also minimizers of the $p$-Dirichlet integral
\begin{deph}
A function $s\in W^{1,p}_{loc}(\Omega)$ is a $p$-supersolution if $\Delta_p h \leq 0$ in the weak sense, i.e.
\begin{gather*}
 \int_{\Omega} \abs{\nabla s}^{p-2}\ps{\nabla s}{\nabla \phi }dV\geq 0 \ \ \ \forall \ \phi \in C^\infty_c(\Omega), \phi \geq 0
\end{gather*}
A function $s:\Omega\to \R\cup\{+\infty\}$ (not everywhere infinite) is said to be $p$-superharmonic if it is lower semicontinuous and for every open $D\Subset \Omega$ and every $p$-harmonic function on $D$ with $h|_{\partial D}\leq s|_{\partial D}$, then $h\leq s$ on all $D$. The space of $p$-superharmonic functions is denoted by $S_p(\Omega)$.
\end{deph}
We recall that all $p$-supersolutions have a lower-semicontinuous rappresentative in $W^{1,p}_{loc}(\Omega)$ and $s\in W^{1,p}_{loc}(\Omega)$ is $p$-superharmonic if and only if it is a $p$-supersolution. In particular thanks to Caccioppoli-type estimates all bounded above $p$-superharmonic functions are $p$-supersolutions. The family of $p$-superharmonic functions is closed under right-directed convergence, i.e. if $s_n$ is an increasing sequence of $p-$superharmonic functions with pointwise limit $s$, then either $s=\infty$ everywhere or $s$ is $p$-superharmonic. By a truncation argument, this also shows that every $p$-superharmonic function is the limit of an increasing sequence of $p$-supersolutions.

Now we turn our attention to special $p$-harmonic functions, the so-called $p$-potentials.
\begin{prop}
 Given $K\subset \Omega\subset R$ with $\Omega$ bounded and $K$ compact, and given $\psi\in C^{\infty}_c(\Omega)$ s.t. $\psi\vert_K=1$, there exists a unique function:
\begin{gather*}
 h\in W^{1,p}(\Omega\setminus K) \ \ \ h-\psi\in W^{1,p}_0(\Omega\setminus K)
\end{gather*}
This function is a minimizer for the $p$-capacity, explicitly:
\begin{gather*}
 \cp(K,\Omega)=\int_\Omega \abs{\nabla h}^p dV
\end{gather*}
for this reason, we call $h$ the $p$-potential of the couple $(K,\Omega)$. Note that if $\Omega$ is not bounded, it is still possible to define its $p$-potential by a standard exhaustion argument. 
\end{prop}
One might ask when the $p$-potential of a couple of sets is continuous on $\overline \Omega$. In this case the set $\Omega\setminus K$ is said to be regular with respect to the $p$-laplacian, or simply $p$-regular. $p$-regularity depends strongly on the geometry of $\Omega$ and $K$, and there exist at least two characterization of this property: the Wiener criterion and the barrier condition. For the aim of this paper we simply note that $p$-regularity is a local property and that if $\Omega\setminus K$ has smooth boundary, then it is $p$-regular. As references for the Wiener criterion and the barrier condition, we cite \cite{3} and \cite{6} (which deal only with $\R^n$, but as observed before local properties of $\R^n$ are easily extended to Riemannian manifolds) and \cite{22}, a very recent article which deals with $p$-harmonicity and $p$-regularity on metric spaces.

Before proceding, we cite some elementary estimates on the capacity.
\begin{lemma}\label{lemma_cap}
Let $K_1\subset K_2\subset \Omega_1\subset \Omega_2\subset R$. Then:
\begin{gather*}
 \cp(K_2,\Omega_1)\geq \cp(K_1,\Omega_1) \ \ \ \cp(K_2,\Omega_1)\geq \cp(K_2,\Omega_2)
\end{gather*}
Moreover, if $h$ is the $p$-potential of the couple $(K,\Omega)$, for $0\leq t<s\leq 1$ we have:
\begin{gather*}
 \cp\ton{\{h\leq s\},\{h<t\}}=\frac{\cp(K,\Omega)}{(s-t)^{p-1}}
\end{gather*}

\end{lemma}
\begin{proof}
The proofs of these estimates follow quite easily from the definitions, and they can be found in propositions 3.6, 3.7, 3.8 in \cite{77}, or in section 2 of \cite{3}. Even though the setting of \cite{3} is $\R^n$, all the argumets used apply also to the Riemannian case.
\end{proof}

In the following section we cite some technical results that will be essential in our proof of the reverse \Ka condition, in particular the solvability of the obstacle problem and the minimizing property of its solutions and a technical lemma about uniformly convex Banach spaces. Section \ref{sec_ka} contains the main results of this article, the proof of the reverse \Ka condition. We tried to use as few technical tools as possible in our proof, so as to make it is readable and understandable by non-specialists.

\section{Obstacle problem}\label{sec_obs}
In this section we present the so-called obstacle problem, a technical tool that will be foundamental in our main theorem.
\begin{deph}
 Let $R$ be a Riemannian manifold and $\Omega\subset R$ be a bounded domain. Given $\theta\in W^{1,p}(\Omega)$ and $\psi:\Omega\to [-\infty,\infty]$, we define the set:
\begin{gather*}
 K_{\theta,\psi}=\{\varphi\in W^{1,p}(\Omega) \ s.t. \ \varphi\geq \psi \ a.e. \ \ \ \varphi-\theta\in W^{1,p}_0(\Omega)\}
\end{gather*}
we say that $s\in K_{\theta,\psi}$ solves the obstacle problem relative to the $p$-laplacian if for any $\varphi\in K_{\theta,\psi}$:
\begin{gather*}
 \int_\Omega \ps{\abs {\nabla s }^{p-2}\nabla s}{\nabla \varphi - \nabla s}dV \geq 0
\end{gather*}
\end{deph}

It is evident that the function $\theta$ defines in the Sobolev sense the boundary values of the solution $s$, while $\psi$ plays the role of obstacle, i.e. $s$ must be $\geq \psi$ at least almost everywhere. Note that if we set $\psi\equiv -\infty$, the obstacle problem turns into the classical Dirichlet problem. Anyway for our purposes the two functions $\theta$ and $\psi$ will always coincide, and in what follows for simplicity we will write $K_{\psi,\psi}\equiv K_{\psi}$.\\
The obstacle problem is a very important tool in nonlinear potential theory, and with the development of calculus on metric spaces it has been studied also in this very general setting. In the following we cite some results relative to this problem and its solvability.
\begin{prop}\label{prop_obs}
 If $\Omega$ is a bounded domain in a Riemannian manifold $R$, the obstacle problem $K_{\theta,\psi}$ has always a unique (up to a.e. equivalence) solution if $K_{\theta,\psi}$ is not empty (which is always the case if $\theta=\psi$). Moreover the lower semicontinuous regularization of $s$ coincides a.e. with $s$ and it is the smallest $p$-superharmonic function in $K_{\theta,\psi}$, and also the function in $K_{\theta,\psi}$ with smallest $p$-Dirichlet integral. If the obstacle $\psi$ is continuous in $\Omega$, then $s\in C(\Omega)$.
\end{prop}
In \cite{3}, Heinonen, Kilpel\"{a}inen and Martio prove this theorem in the setting of a Euclidean measure space with a doubling property and a Poincaré inequality, but it is quite clear that the techniques involved also apply to the setting of any Riemannian manifold. As mentioned before, this problem has been extensively studied also on measure metric spaces with a doubling property and a Poincaré inequality (for example bounded domains in Riemannian manifolds with respect to the measure induced by the metric), and proposition \ref{prop_obs} holds even in this more general setting (see \cite{21}).

We make a remark on the minimizing property of the $p$-superharmonic function, which will play a central role in our proof.
\begin{rem}\label{rem_min}
 Let $s$ be a $p$-superharmonic function in $W^{1,p}(\Omega)$ where $\Omega\subset R$ is a bounded domain. Then for any function $f\in W^{1,p}(\Omega)$ with $f\geq s$ a.e. and $f-s\in W^{1,p}_0(\Omega)$ we have:
\begin{gather*}
 \D_p(s)\leq \D_p(f)
\end{gather*}
\end{rem}
\begin{proof}
 This remark follows easily form the minimizing property of the solution to the obstacle problem. In fact, the previous proposition shows that $s$ is the solution to the obstacle problem relative to $K_s$, and the minimizing property follows.
\end{proof}

When it comes to the obstacle problem (or similarly to the Dirichlet problem), the regularity of the solution on $\partial \Omega$ is always a good question. An easy corollary to theorem 7.2 in \cite{22} is the following:
\begin{prop}\label{prop_cont}
 Given a bounded $\Omega\subset R$ with smooth boundary and given a $\psi\in W^{1,p}(\overline \Omega)\cap C(\overline \Omega)$, then the unique solution to the obstacle problem $K_\psi$ is continuous up to $\partial \Omega$.
\end{prop}
Note that it is not necessary to assume $\partial \Omega$ smooth, it suffices to assume $\partial \Omega$ regular with respect to the $p$-Dirichlet problem, or equivalently that satifies the Wiener criterion in each point, but we think that for the aim of this paper it is not necessary to go into such interesting but quite technical details.

In the following we will need this lemma about uniform convexity in Banach spaces. This lemma doesn't seem very intuitive at first glance, but a very simple two dimensional drawing of the vectors involved shows that in fact it is quite natural.
\begin{lemma}\label{lemma_*}
 Given a uniformly convex Banach space $E$, there exists a function $\sigma:[0,\infty)\to [0,\infty)$ strictly positive on $(0,\infty)$ with $\lim_{x\to 0} \sigma(x)=0$ such that for any $v,w\in E$ with $\norm{v+1/2 w}\geq \norm v$:
\begin{gather*}
 \norm{v+w}\geq \norm v \ton{1+\sigma\ton{\frac{\norm w}{\norm v +\norm w}}}
\end{gather*}
\end{lemma}
\begin{proof}
 Note that by the triangle inequality $\norm{v+1/2 w}\geq \norm v$ easily implies $\norm{v+ w}\geq \norm v$. Let $\delta$ be the modulus of convexity of the space $E$. By definition we have:
\begin{gather*}
 \delta(\epsilon)\equiv \inf\left\{ 1-\norm{\frac{x+y}{2}} \ s.t. \ \norm x, \norm y \leq 1 \ \ \ \norm {x-y}\geq \epsilon \right\}
\end{gather*}
Consider the vectors $x=\alpha v$ $y=\alpha (v+w)$ where $\alpha=\norm{v+w}^{-1}\leq \norm{v}^{-1}$. Then:
\begin{gather*}
 1-\norm{\frac{x+y}{2}}=1-\alpha \norm{v+\frac{w}{2}}\geq \delta(\alpha \norm{w})\geq \delta\ton{\frac{\norm w}{\norm v + \norm w}}\\
\norm{v+w}\geq \norm{v+\frac{w}{2}} \ton{1-\delta\ton{\frac{\norm w}{\norm v + \norm w}}}^{-1}
\end{gather*}
Since $\norm{v+\frac{w}{2}}\geq \norm v$ and by the positivity of $\delta$ on $(0,\infty)$ if $E$ is uniformly convex, the thesis follows.
\end{proof}

Recall that all $L^p(X,\mu)$ spaces with $1<p<\infty$ are uniformly convex thanks to Clarkson's inequalities, and their modulus of convexity is a function that depends only on $p$ and not on the underling measure space $X$. For a reference on uniformly convex spaces, modulus of convexity and Clarkson's inequality, we cite his original work \cite{7}.

\section{\Ka condition}\label{sec_ka}
In this section, we prove the \Ka condition for a generic $p>1$ and show that it is not just a sufficient condition, but also a necessary one.
\begin{prop}[\Ka condition]
 If there exists a compact set $K\subset R$ and a $p$-superharmonic finite-valued function $\ka:R\setminus K\to \R$ with
\begin{gather*}
 \lim_{x\to \infty} \ka(x)=\infty
\end{gather*}
then $R$ is $p$-parabolic.
\end{prop}
\begin{proof}
 This condition was proved in \cite{2} in the case $p=2$, however since the only tool necessary for this proof is the comparison principle, it is easily extended to any $p>1$. An alternative proof can be found in \cite{99}.\\
Fix an open relatively compact set $D$ with $K\subset D$ (for simplicity, we may also assume $\partial D$ smooth), and fix an exhaustion $D_n$ of $R$ with $D_0\equiv D$. Set $m_n\equiv \min_{x\in \partial D_n} \ka(x)$, and consider for every $n$ the $p$-capacity potential $h_n$ of the couple $(\overline D,D_n)$. Since $\ka$ is superharmonic, it is easily seen that $h_n(x)\geq 1-\ka(x)/m_n$ for all $x\in D_n\setminus \overline D$. By letting $n$ go to infinity, we obtain that $h(x)\geq 1$ for all $x\in R$, where $h$ is the capacity potential of $(\overline D, R)$. Since by the maximum principle $h(x)\leq 1$ everywhere, $h(x)=1$ and so $\cp(\overline D)=0$.
\end{proof}

Observe that the hypothesis of $\ka$ being finite-valued can be dropped. In fact if $\ka$ is $p$-superharmonic, the set $\{x \ s.t. \ \ \ka(x)=\infty\}$ has $p$-capacity zero, and so the reasoning above would lead to $h(x)= 1$ except on a set of $p$-capacity zero, but this indeed implies $h(x)=1$ everywhere (see \cite{3} for the details).

Before proving the reverse of \Ka condition for any $p>1$, we present a short simpler proof in the case $p=2$ and we briefly describe the reasoning that brought us to the general proof. In the linear case, the sum of $2$-superharmonic functions is again $2$-superharmonic, but of course this fails to be true for a generic $p$. Thanks to linearity, it is easy to prove that:
\begin{prop}
Given a $2$-parabolic Riemannian manifold, for any compact set $K$ with smooth boundary (actually $p$-regular is enough), there exists a $2$-superharmonic continuous function $\ka:R\setminus K\to \R^+$ with $f|_{\partial K}=0$ and $\lim_{x\to \infty}\ka(x)=\infty$.

\end{prop}
\begin{proof}
Consider a regular (=with smooth boundary) exhaustion $\{K_n\}_{n=0}^\infty$ of $R$ with $K_0\equiv K$. For any $n\geq1$ define $h_n$ to be the $p$-potential of $(K,K_n)$. By the comparison principle, the sequence $\tilde h_n =1-h_n$ is a decreasing sequence, and since $R$ is $2$-parabolic the limit function $\tilde h$ is the zero function. By Dini's theorem, the sequence $\tilde h_n$ converges to zero locally uniformly, so it is not hard to choose a subsequence $\tilde h_{n(k)}$ such that the series $\sum_{k=1}^\infty \tilde h_{n(k)}$ converges locally uniformly to a continuous function. It is straightforward to see that $\ka=\sum_{k=1}^{\infty} \tilde h_{n(k)}$ has all the desidered properties.
\end{proof}

For the nonlinear case, even though this proof doesn't apply, the idea is similar in some aspects. Indeed, we will build an increasing locally uniformly bounded sequence of $p$-superharmonic functions, and the limit of this sequence will be the function $\ka$.

The idea behind the proof in the nonlinear case is to extend the following well-known result about sets of $p$-capacity zero.

\begin{prop}\label{prop_mimic}
A set $E\subset \R^n$ is of $p$-capacity zero if and only if there exists a $p$-superharmonic $s$ function with $s|_E=\infty$.
\end{prop}
\begin{proof}
 See theorem 10.1 in \cite{3} for the proof.
\end{proof}

Consider the $p$-Royden compactification $R_p^*$ of the manifold $R$ (like every compactification, the boundary $\Gamma_p=R^*_p\setminus R$ reflects in some sense the behaviour of $R$ at infinity). The concept of $p$-capacity can be extended to subsets of $R^*_p$, and it turns out that $R$ is $p$-parabolic if and only of $\cp(\Gamma_p)=0$ (see for example \cite{11}). Then in some sense, by mimicking the proof of proposition \ref{prop_mimic}, we get our statement. There are although some tecnical aspects to be considered, for example the boundedness assumption on the domain $\Omega$ makes it impossible to use the theory of the obstacle problem to solve it on the complement of a compact set in $R$, and also some convergence properties of the solutions are not so obvious and need some careful consideration.

For the sake of simplicity, in this article we chose to limit the use of abstract technical tools like the $p$-Royden compactification and follow instead a more direct approach.

We first prove that if $R$ is $p$-parabolic, then there exists a proper function $f:R\to \R$ with finite $p$-Dirichlet integral.
\begin{prop}
 Let $R$ be a $p$-parabolic Riemannian manifold. Then there exists a positive continuous function $f:R\to \R$ such that:
\begin{gather*}
 \int_{R} \abs{\nabla f}^p dV <\infty \ \ \ \ \ \ \lim_{x\to \infty} f(x)=\infty
\end{gather*}
\end{prop}
\begin{proof}
 Fix an exhaustion $\{D_n\}_{n=0}^{\infty}$ of $R$ such that every $D_n$ has smooth boundary, and let $\{h_n\}_{n=1}^{\infty}$ be the $p$-capacity potential of the couple $(D_0,D_n)$. Then by an easy application of the comparison principle the sequence:
\begin{gather*}
 \tilde h_n(x) \equiv \begin{cases}
                    0 & \text{if }x\in D_0\\
1-h_n(x) & \text{if }x\in D_n\setminus D_0\\
1 & \text{if }x\in D_n^C
                   \end{cases}
\end{gather*}
 is a decreasing sequence of continuous function converging pointwise to $0$ (and so also locally uniformly by Dini's theorem) and also $\int_{R}\abs{\nabla \tilde h_n}^p dV \to 0$. So we can extract a subsequence $\tilde h_{n(k)}$ such that
\begin{gather*}
 0\leq \tilde h_{n(k)}(x)\leq \frac{1}{2^k} \ \ \ \ \forall x\in D_k \ \ \ \ \wedge \ \ \ \ \ \int_{R}\abs{\nabla \tilde h_{n(k)}}^p dV<\frac1 {2^k}
\end{gather*}
It is easily verified that $f(x)=\sum_{k=1}^{\infty} \tilde h _{n(k)}(x)$ has all the desidered properties.
\end{proof}

 We are now ready to prove the reverse $\Ka$ condition, i.e.:
\begin{teo}
Given a $p$-parabolic manifold $R$ and an open nonempty compact $K\subset R$ with smooth boudary, there exists a continuous positive superharmonic function $\ka:R\setminus \overline K\to \R$ such that
\begin{gather*}
 \lim_{x\to \infty} \ka(x)=\infty
\end{gather*}
\end{teo}
\begin{proof}
 Fix a continuous proper function $f:R\to \R^+$ with finite Dirichlet integral such that $f=0$ on a compact neighborhood of $K$, and let $D_n$ be a smooth exhaustion of $R$ such that $f|_{D_n^C}\geq n$. We want to build by induction an increasing sequence of continuous functions $s^{(n)} \in L^{1,p}(R)$ $p$-superharmonic in $\overline{K}^C$ with $s^{(n)}|_K=0$ and such that $s^{(n)}=n$ in a neighborhood of infinity (say $S_n^C$, where $S_n$ is compact). Moreover we will ask that $s^{(n)}$ is locally uniformly bounded, so that $\ka(x)\equiv \lim_n s^{(n)}(x)$ is finite in $R$ and has all the desidered properties.\\
Let $s^{(0)}\equiv 0$, and suppose by induction that an $s^{(n)}$ with the desidered property exists. Hereafter $n$ is fixed, so for simplicity we will write $s^{(n)}\equiv s$, $s^{(n+1)}\equiv s^+$ and $S_n = S$. Define the functions $ f_j(x)\equiv \min\{j^{-1}f(x),1\}$, and consider the obstacle problems on $\Omega_j\equiv D_{j+1}\setminus \overline D_0$ given by the obstacle $\psi_j=s+ f_j$ \footnote{Note that $\psi_j\in L^{1,p}_0(R)$, in fact for every $j$ $\psi_j=n+1$ in a neighborhood of infinity}. For any $j$, the solution $h_j$ to this obstacle problem is a $p$-superharmonic function defined on $\Omega_j$ bounded above by $n+1$ and whose restriction to $\partial D_0$ is zero. If $j$ is large enough such that $s=n$ on $D_{j}^C$ (i.e. $S\subset D_j$), then the function $h_j$ is forced to be equal to $n+1$ on $D_{j+1}\setminus D_{j}$ and so the function:
\begin{gather*}
 \tilde h_j(x)\equiv \begin{cases}
                   h_j(x) & x\in \Omega_j\\
0 & x\in \overline D_0\\
n+1 & x\in D_{j+1}^C
                  \end{cases}
\end{gather*}
is a continuous function on $R$, $p$-superharmonic in $\overline{D_0}^C$. If we are able to show that $\tilde h_j$ converges locally uniformly to $s$, then we can choose an index $\bar j$ large enough to have $\sup_{x\in D_{n+1}} \abs{\tilde h_{\bar j} (x)-s(x)}<2^{-n-1}$, and so the function $s^+=\tilde h_{\bar j}$ has all the desidered properties.

For this aim, consider $\delta_j \equiv h_{j}-s$. Since the sequence $h_{j}$ is decreasing thanks to the properties of the solution to the obstacle problems, so is $\delta_j$ and therefore it converges pointwise to a function $\delta\geq 0$. By the minimizing properties of $h_j$, we have that
\begin{gather*}
 \norm{\nabla h_j}_p\leq \norm{\nabla s + \nabla f_j}_p\leq \norm{\nabla s}_p + \norm{\nabla f}_p\\
\norm{\nabla \delta_j}_p\leq 2\norm{\nabla s} + \norm{\nabla f}\leq C
\end{gather*}
and a standard weak-compactness argument in reflexive spaces shows that $\delta\in L^{1,p}_0(R)$ with $\nabla \delta_j \to \nabla \delta$ in the weak $L^p$ sense (see for example lemma 1.33 in \cite{3}). Now we prove that $\D_p(\delta_j)\to 0$ so that $\D_p(\delta)=0$, and since $\delta=0$ on $D_0$ we conclude $\delta=0$. Note also that since the limit function $\delta$ is continuous, Dini's theorem assures that the convergence is locally uniform.

Let $\lambda>0$ (for example $\lambda = 1/2$), and consider the function $g(x)\equiv \min\{s+\lambda \delta_j,n\}$. It is quite clear that $s$ is the solution to the obstacle problem relative to itself on $S\setminus \overline {D_0}$, and since $\delta_j\geq0$ with $\delta_j=0$ on $D_0$, $g\geq s$ and $g-s\in W^{1,p}_0(S\setminus \overline {D_0})$. The minimizing property for solutions to the $p$-laplace equation then guarantees that:
\begin{gather*}
\norm{\nabla s + \lambda \nabla \delta_j}_p^p\equiv \int_{R} \abs{\nabla s + \lambda \nabla \delta_j}^p dV \geq \int_{S\setminus \overline{D_0}} \abs{\nabla g}^p dV\geq\\
\geq \int_{S\setminus \overline{D_0}} \abs{\nabla s}^p dV = \norm{\nabla s}_p^p
\end{gather*}
Recalling that also $h_j=s+\delta_j$ is solution to an obstacle problem on $D_{j+1}\setminus D_0$, we get:
\begin{gather*}
 \norm{\nabla s + \nabla \delta_j}_p=\ton{\int_{R} \abs{\nabla \tilde h_j}^p dV}^{1/p}=\ton{\int_{\Omega_j} \abs{\nabla h_j}^p dV}^{1/p}\leq \\
\leq \ton{\int_{\Omega_j} \abs{\nabla s + \nabla f_j}^p dV}^{1/p}\leq \norm{\nabla s +\nabla f_j}_p \leq \norm{\nabla s }_p + \norm{\nabla {f_j}}_p
\end{gather*}
Uding lemma \ref{lemma_*} we conclude:
\begin{gather*}
 \norm{\nabla s}_p \ton{1+\sigma\ton{\frac{\norm{\nabla \delta_j}_p}{ \norm{\nabla s}_p +\norm{\nabla \delta_j}_p}}}\leq \norm{\nabla s}_p + \norm{\nabla f_j}_p
\end{gather*}
Since $\norm{\nabla f_j}_p \to 0$ as $j$ goes to infinity and by the properties of the function $\sigma$:
\begin{gather*}
\lim_{j\to \infty}\D_p(\delta_j)\equiv \lim_{j\to \infty} \norm{\nabla \delta_j}_p^p =0
\end{gather*}

\end{proof}

\begin{rem}
 \rm Since $\norm{\nabla \tilde h_j}_p\leq \norm{\nabla s^{(n)}}_p+\norm{\nabla \delta_j}_p$, if for each induction step we choose $\bar j$ such that $\norm{\nabla \delta_{\bar j}}_p<2^{-n}$, the function $\ka=\lim_{n} s^{(n)}$ has finite $p$-Dirichlet integral.
\end{rem}

\begin{rem}
 \rm In the previous theorem we built a function $\ka$ which is proper and continuous in $R$, $p$-superharmonic in $K^C$ and zero on $K$ assuming $K$ compact with smooth boundary and with non-empty interior. However it is clear from the proof that these assumptions can be weakened. In fact, the only properties we need are that if a function $\delta$ is constant and zero on $K$, than it has to be zero everywhere on $R$, and the obstacle problem relative to $\D_j\setminus K$ has to be solvable with continuity on the boundary. From these we notice that it is sufficient to assume $K$ $p-$regular, which implies also that $cap_p(K,\D_j)>0$ and so $\delta=0$.
\end{rem}

\section{Evans potentials}
We conclude this work with some remarks on the Evans potentials for $p$-parabolic manifolds. Given a compact set with nonempty interior and smooth boundary $K\subset R$, we call $p$-Evans potential a function $\E:R\setminus K\to \R$ $p$-harmonic where defined such that:
\begin{gather*}
 \lim_{x\to \infty } \E(x)=\infty \ \ \ \ \ \lim_{x\to \partial K} \E(x)=0
\end{gather*}
It is evident that if such a function exists, then the \Ka condition guarantees the $p$-parabolicity of the manifold $R$. It is interesting to investigate whether also the reverse implication holds. In \cite{19} and \cite{5}, Nakai and Sario prove that $2$-parabolicity of Riemannian surfaces is completely characterized by the existence of such functions. In particular they prove that:
\begin{teo}\label{teo_sn}
Given a $p$-parabolic Riemannian surface $R$, and an open precompact set $R_0$, there exists a ($2-$)harmonic function $\E:R\setminus R_0\to \R^+$ which is zero on the boundary of $R_0$ and goes to infinity as $x$ goes to infinity. Moreover:
\begin{gather}\label{eq_evp}
 \int_{\{0\leq \E(x)\leq c\}} \abs{\nabla \E(x)}^2dV\leq 2\pi c
\end{gather}
\end{teo}
This is the content of \cite{19} and theorems 12.F and 13.A in \cite{5}. Clearly the constant $2\pi$ in equation \ref{eq_evp} can be substituted by any other positive constant. As noted in the Appendix to \cite{5} (in particular pag. 400), with similar arguments and with the help of the classical potential theory (\cite{37} might be of help in some technical details), it is possible to prove the existence of $2$-Evans potentials for a generic $n$-dimensional $2$-parabolic Riemannian manifold.

This argument however is not adaptable to the nonlinear case ($p\neq 2$). In fact it relies heavily on the harmonicity of Green potentials and on tools like the energy and transfinite diameter of a set that are not available in the nonlinear contest. In the end the potential $\E$ is build as a special convex combination of Green kernels defined on the $2$-Royden compactification of $R$, and while convex combinations preserve $2$-harmonicity, this is evidently not the case when $p\neq 2$.

Since $p$-harmonic functions minimize the $p$-Dirichlet of functions with the same boundary values, it would be interesting from a theoretical point of view to prove existence of $p$-Evans potentials and maybe also to determine some of their properties. From the practical point of view such potentials could be used to get informations on the underlying manifold $R$, for example they can be used to improve the Kelvin-Nevanlinna-Royden criterion for $p$-parabolicity as shown in \cite{1}.

Even though we were not able to prove the existence of such potentials in the generic case, some particular cases are easier to manage. As shown in \cite{99}, conclusions similar to the ones in theorem \ref{teo_sn} can be easily obtained in the case $R$ is a model manifold or all of its ends are roughly Euclidean or Harnack. We briefly discuss these very particular cases hoping that the ideas involved in these proofs will be a good place to start for a proof in the general case.

First of all we recall the definition model manifolds:
\begin{deph}
A complete Riemannian manifold $R$ is a model manifold (or a spherically simmetric manifold) if it is diffeomorphic to $\R^n$ and if there exists a point $o\in R$ such that in exponential polar coordinates the metric assumes the form:
\begin{gather*}
 g_{ij}=\begin{Bmatrix}
         1 &0\\ 0& \sigma(r)\delta_{ij}
        \end{Bmatrix}
\end{gather*}
where $\sigma$ is a smooth positive function on $(0,\infty)$ with $\sigma(0)=0$ and $\sigma'(0)=1$.
\end{deph}
For some references on polar coordinates and model manifolds, we cite \cite{20} (a very complete survey on $2$-parabolicity) and the book \cite{4}.

Define the function $A(r)=\sqrt{g(r)}=\sigma(r)^{\frac{n-1}2}$, where $g(r)$ is the determinant of the metric tensor. Note that, except for a constant depending only on $n$, $A(r)$ is the area of the sphere of radius $r$. On model manifolds, the radial function
\begin{gather*}
 f_{p,\bar r}(r)\equiv\int_{\bar r} ^r A(t)^{-\frac 1 {(p-1)}} dt
\end{gather*}
is a $p$-harmonic function away from the origin $o$, in fact:
\begin{gather*}
 \Delta_p (f)=\frac{1}{\sqrt g} \operatorname{div}(\abs{\nabla f}^{p-2}\nabla f)= \frac{1}{\sqrt g} \partial_i \ton{\sqrt g \ton{g^{kl}\partial_k f \partial_l f}^{\frac{p-2} 2} g^{ij}\partial_j f}=\\
=\frac{1}{ {A(r)} }\partial_r \ton{A(r)\ A(r)^{-\frac{p-2}{p-1}}\ A(r)^{-\frac{1}{p-1}}\ \hat r}=0
\end{gather*}
where $\hat r$ is the gradient of $r$, which in polar coordinates has components $(1,0\cdots,0)$. The function $\min\{f_{p,\bar r},0\}$ is a $p$-subharmonic function on $R$, so if $f_{p,\bar r}(\infty)<~\infty$, $R$ cannot be $p$-parabolic. A straightforward application of the \Ka condition shows that also the reverse implication holds, so that a model manifold $R$ is $p$-parabolic if and only if $f_{p,\bar r}(\infty)=\infty$. This shows that if $R$ is $p$-parabolic, then for any $\bar r>0$ there exists a radial $p$-Evans potential $f_{p,\bar r}\equiv\E_{\bar r}:R\setminus B_{\bar r}(0)\to \R^+$, moreover it is easily seen by direct calculation that:
\begin{gather*}
 \int_{B_R} \abs{\nabla \E_{\bar r}}^{p}dV=\E_{\bar r}(R) \ \ \ \ \Longleftrightarrow \ \ \ \  \int_{\E_{\bar r}\leq t} \abs{\nabla \E_{\bar r}}^p dV =t
\end{gather*}
This estimate is similar to the one in equation \ref{eq_evp}, and it allow us to conclude that:
\begin{gather*}
 \cp(B_{\bar r}, \{\E_{\bar r}\leq t\})=\int_{\{\E_{\bar r}\leq t\}\setminus B_{\bar r}} \abs{\nabla\ton{\frac{\E_{\bar r}}{t}}}^p dV =t^{1-p}
\end{gather*}
Since $R$ is $p$-parabolic, it is clear that $ \cp(B_{\bar r}, \{\E_{\bar r}\leq t\})$ must go to $0$ as $t$ goes to infinity, but this estimate tells us also how fast the convergence is.

What we want to show now is that $p$-parabolic model manifolds admit $p$-Evans potentials relative to any compact set $K$.
\begin{prop}
 Let $R$ be a $p$-parabolic model manifold and $K\subset R$ a $p$-regular compact set. Then there exists an Evans potential $e:R\setminus K \to \R^+$ with
\begin{gather}\label{eq_evp2}
 \cp(K,\{e<t\})\sim t^{1-p}
\end{gather}
as $t$ goes to infinity.
\end{prop}
\begin{proof}
 Since $K$ is bounded, there exists $\bar r>0$ such that $K\subset B_{\bar r}$. Let $\E_{\bar r}$ be the radial $p$-Evans potential relative to this ball. For any $n>0$, set $A_n=\{\E_{\bar r}\leq n\}$ and define the function $e_n$ to be the unique $p$-harmonic function on $A_n\setminus K$ with boundary values $n$ on $\partial A_n$ and $0$ on $\partial K$. An easy application of the comparison principle shows that $e_n\geq \E_{\bar r}$ on $A_n\setminus K$, and so the sequence $\{e_n\}$ is increasing. By the Harnack principle, either $e_n$ converges locally uniformly to a harmonic function $e$, or it diverges everywhere to infinity. To exclude the latter, set $m_n$ to be the minumum of $e_n$ on $\partial B_{\bar r}$. By the maximum principle the set $\{0\leq e_n \leq m_n\}$ is contained in the ball $B_{\bar r}$, and using the capacity estimates in \ref{lemma_cap}, we get that:
\begin{gather*}
 \cp(K,B_{\bar r})\leq \cp(K, \{e_n< m_n\})=\cp(K,\{e_n/n < m_n/n\}) =\\
=\frac{n^{p-1}}{m_n^{p-1}}\cp{K, e_n<n} \leq \frac{n^{p-1}}{m_n^{p-1}}\cp(B_{\bar r},\E_{\bar r}<n)\\
\\
m_n^{p-1} \leq \frac{n^{p-1}\cp(B_{\bar r}, \{\E_{\bar r}< n\})}{\cp(K,B_{\bar r})}<\infty
\end{gather*}
So the limit function $e=\lim_n e_n$ is a $p$-harmonic function in $R\setminus K$ with $e\geq \E_{\bar r}$. Boundary continuity estimates like the one in \cite{78} (p236) prove that $e|_{\partial K}=0$, but in this case we can use a more simple argument. Let in fact $M$ be the maximum of $e$ on $\partial B_{\bar r}$. Then by the comparison principle, $0\leq e_n\leq M h$ for every $n$, where $h$ is the $p$-harmonic potential of $(K,B_{\bar r})$. The $p$-regularity of $K$ ensures that $h$ is continuous up to the boudary with $h|_{\partial K}=0$, and so the claim is proved.

To prove the estimates on the capacity, consider that by the comparison principle for every $n$ (and so also for the limit) $e_n \leq M+\E_{\bar r}$ (where this relation makes sense), so that:
\begin{gather*}
 \cp(K,\{e<t\})\leq \cp(B_{\bar r},\{e<t\})\leq \\
\leq\cp(B_{\bar r}, \{\E_{\bar r}<t-M\}) =(t-M)^{1-p}\sim t^{1-p}
\end{gather*}
For the reverse inequality, we have:
\begin{gather*}
 \cp(K, \{e<t\})=\int_{\{e<t\}\setminus K}\abs {\nabla \ton{\frac e t}}^p dV =\int_{\{e<M\}\setminus K} \abs {\nabla \ton{\frac e t}}^p dV + \\
+\int_{\{M<e<t\}} \abs {\nabla \ton{\frac e t}}^p dV =\ton{\frac m t }^p\int_{\{e<M\}\setminus K} \abs {\nabla \ton{\frac e M}}^p dV +\\
+\ton{\frac {t-m}{t}} ^p\int_{\{M<e<t\}} \abs {\nabla \ton{\frac e {t-M}}}^p dV \geq\\
\geq\ton{\frac m t }^p \cp(K, \{e<M\})+ \ton{\frac {t-m}{t}} ^p \cp{B_{\bar r}, \E_{\bar r}}\sim t^{1-p}
\end{gather*}

\end{proof}

We conclude this work with a very simple consideration. Lemma 2.14 in \cite{8} proves that on the complement of a $p$-regular compact set in a $p$-parabolic Riemannian manifold, there always exists a positive unbounded harmonic function $\E'$. On some particular manifolds, every nonnegative $p$-harmonic function has a limit at infinity, and in particular any unbounded function goes to infinity on the ideal boundary of $R$. These shows that the function $\E'$ is actually an Evans potential.\\
For example, if $R$ is roughly Euclidean or if all of its ends are Harnack ends, then all nonnegative $p$-harmonic functions have a limit at infinity. As references for these particular properties we cite \cite{8} and \cite{9}, in particular lemma 3.23 in \cite{9} and paragraph 3 in \cite{8}.

Even though we weren't able to prove that every $p$-parabolic manifold admit a $p$-Evans potential, the partial proofs of this last section suggest that this is true, or at least that investigating this problem could be interesting.

Note that Evans potentials and in particular estimates like \ref{eq_evp} and \ref{eq_evp2} are useful to study the behaviour of functions on the manifold $R$, as proved in \cite{1}.

\end{document}